\documentclass[]{article}
\usepackage[utf8]{inputenc}
\usepackage{amsmath,amssymb,amsthm,mathtools}
\usepackage{graphicx}

\newtheorem{defn}{Definition}
\newtheorem{thm}{Theorem}

\newtheorem{cor}{Corollary}
\newtheorem{prop}{Proposition}

\title{Approximating Continuous Functions with Scattered Translates of the Poisson Kernel}
\author{Jeff Ledford}
\date{March 2013}

\begin{document}

\maketitle

\section{Introduction}
In \cite{Powell}, it was shown that continuous functions on a closed interval may be uniformly approximated by scattered translates of the Hardy multiquadric.  We will adapt the method found there to our purposes, showing that the same is true for the Poisson kernel, $\phi(x)=(a^2+x^2)^{-1}$.

\noindent This note is organized as follows.  In the next section, various definitions and facts are collected.  The third section contains the main theorem to be proved, while the fourth section contains the details of the proof.

\section{Definitions and Basic Facts}
We will need to know what "scattered" means.  For our purposes, we have the following definition in mind.
\begin{defn}
A sequence of real numbers, denoted $\mathcal{X}$, is said to be $\delta$-separated if
\[
\inf_{\overset{x,y\in\mathcal{X}}{x\neq y}}|x-y|= \delta >0
\]
\end{defn}
\noindent It's not hard to see that a $\delta$-separated sequence must be countable.  Take intervals of lenth $\delta/3$ centered at each point in $\mathcal{X}$, each of these intervals is disjoint and contains a rational number $r$.  Letting a member of $\mathcal{X}$ corrspond to the number $r$ which is in the same interval shows that the set $\mathcal{X}$ is at most countable.  This allows us to index $\mathcal{X}$ with the integers.
\begin{defn}
A sequence $\{x_j\}\subset \mathbb{R}$ is scattered if it is $\delta$-separated for some positive $\delta$ and satisfies 
\[
\lim_{j\to\pm\infty}x_j=\pm\infty
\]
\end{defn}
\noindent Throughout the remainder of the paper we let $\mathcal{X}=\{x_j\}_{j\in\mathbb{Z}}$ be a fixed but otherwise arbitrary scattered sequence.

\section{The Main Result}
\begin{thm}
Given a scattered sequence $\{x_j\}$, $\epsilon > 0 $, and a continuous function $f:[a,b]\to\mathbb{R}$, we may find a sequence of coefficients $\{a_j\}_{j=1}^{N}$, such that
\[
\sup_{x\in [a,b]}\left|f(x) - \sum_{j=1}^{N}\dfrac{a_j}{\alpha^2+(x-x_j)}     \right| < \epsilon 
\]
\end{thm}
\begin{proof}[Sketch of Proof]  The idea is to develop a Taylor expansion
\[
\dfrac{1}{\alpha^2+(x-x_j)^2} = \dfrac{1}{x_j^2}\left[A_0(x)+\dfrac{A_1(x)}{x_j}+\dfrac{A_2(x)}{x_j^2}+  \cdots\right].
\]
From here we show that the linear span of $\{A_j(x)\}$ contains $x^j$ for $j=0,1,2,\dots$.  We then find coefficients to approximate an $n$-th degree polynomial by using an appropriate Vandermonde matrix.  Finally, since we may approximate polynomials, we appeal to the Stone-Weierstrass Theorem to finish our problem.
\end{proof}
\noindent This theorem combined with H\"{o}lder's Inequality lets us replace the $L^{\infty}([a,b])$ norm above with the $L^p([a,b])$ norm.  We state this in the following corollary.
\begin{cor}
Given a scattered sequence $\{x_j\}$, $\epsilon > 0 $, $p\in [1,\infty]$, and a continuous function $f:[a,b]\to\mathbb{R}$, we may find a sequence of coefficients $\{a_j\}_{j=1}^{N}$, such that
\[
\left\|f(x) - \sum_{j=1}^{N}\dfrac{a_j}{\alpha^2+(x-x_j)}\right\|_{L^p([a,b])} < \epsilon 
\]
\end{cor}

\section{Details}
This section provides a rigorous justification for the outline of the proof. we begin with the Taylor expansion.  For any nonzero $x_j$ we have,
\begin{align*}
&\dfrac{1}{\alpha^2+(x-x_j)^2} = \dfrac{1}{x_j^2} \left[1+\dfrac{-2x}{x_j}+\dfrac{x^2+\alpha^2}{x_j^2}    \right]^{-1}\\
=&\dfrac{1}{x_j^2}\sum_{n=0}^{\infty}\dfrac{A_n(x)}{x_j^n}.
\end{align*}
\noindent This leads to the following relationship for $x_j>>0$
\begin{align}
1=&\left[1+\dfrac{-2x}{x_j}+\dfrac{x^2+\alpha^2}{x_j^2}    \right]\sum_{n=0}^{\infty}\dfrac{A_n(x)}{x_j^n} \nonumber \\
=& A_0(x)+\dfrac{A_1(x)-2xA_0(x)}{x_j}\nonumber \\
&\qquad+\sum_{n=2}^{\infty}\dfrac{ A_n(x)-2xA_{n-1}(x)+(x^2+\alpha^2)A_{n-2}(x)}{x_j^n} \label{eq1}
\end{align}
In solving \eqref{eq1} we can see that $A_n(x)$ satisfies the recursion relationship:
\begin{align}
&A_0(x)=1 \nonumber \\
&A_1(x)=2x    \nonumber\\
&A_n(x)=2xA_{n-1}(x)-(x^2+\alpha^2)A_{n-2}(x);\quad n\geq 2 \label{recur}
\end{align}

We are in position to state our first proposition.
\begin{prop}
The leading term of $A_n(x)$ is given by $(n+1)x^n$.
\end{prop}
\begin{proof}
We induct on $n$.  The first two cases are shown above, so we suppose that the assertion holds for all $k$ such that $1\leq k \leq n$.  From \eqref{recur}, we have
\[
A_{n+1}(x)=2xA_n(x)-(x^2+\alpha^2)A_{n-1}(x)
\]
The leading term is calculated using the leading terms of $A_n(x)$ and $A_{n-1}$.  This leads to
\[
2x(n+1)x^n-x^2(nx^{n-1})=[2n+2-n]x^{n+1}=(n+2)x^{n+1}
\]
This is the desired result.
\end{proof}
\noindent The goal of this calculation is the following.
\begin{cor}
The set $\{A_n(x)\}_{n=0}^{\infty}$ is linearly independent on $[a,b]$.
\end{cor}
\noindent From this, we have that $\Pi[x]\subset \text{span}\{A_n(x)\}_{n=0}^{\infty}$, where 
\[
\Pi[x]=\{\text{polynomials in $x$ with coefficients in } \mathbb{R} \}.
\]
We need a way to produce a specific polynomial.  To this end, we choose a subsequence of $\{x_j\}$ as follows.  Let $x_{j(1)}>>0$, then choose each subsequent term according to $x_{j(n+1)} \geq 2x_{j(n)} $, this is possible since $x_j\to\infty$.

We use the following.
\begin{prop}
The following matrix is invertible
\[
P_N=\left[ x_j(k)^{-(l+1)}\right]_{l,k} \qquad l,k=1,2,\dots,N.
\]
\end{prop}
\begin{proof}
We notice that this is a variant of a Vandermonde matrix whose determinant is given by
\[
\det({P_N})=\displaystyle\prod_{k=1}^{N}x_{j(k)}^{-2}\displaystyle\prod_{1\leq r<s\leq N}\left[ \dfrac{1}{x_{j(s)}}-\dfrac{1}{x_{j(r)}} \right],
\]
which is nonzero by our choice of subsequence since $x_{j(r)}\neq x_{j(s)}$ unless $r=s$.
\end{proof}
\begin{prop}Let $N \geq 1$, then the matrix equation
\[
P_N \bf{b}_N = {e_N},
\]
where $\bf{e_N}$ is the $N-$th standard basis vector in $\mathbb{R}^N$,  has solution
\begin{equation}\label{coeff}
{\bf b_{N}}(m)=(-1)^{N+m}x_{j(m)}^{N+1}\prod_{k \neq m}\left[1-\dfrac{x_{j(m)}}{x_{j(k)}}   \right]^{-1}\quad m=1,\dots,N.
\end{equation}
\end{prop}
\begin{proof}
In this case, Cramer's Rule is easy to work with since it leaves us with the ratio of Vandermonde determinants.  If we set $P_N(m)$ to be the matrix $P_N$ with the $m$-th column replaced by $\bf{e_N}$, then we have
\[
{\bf b_{N}}(m)=\dfrac{\det (P_N(m)) }{\det(P_N)}
\]
We need only work out $\det (P_N(m))$ and simplify.
\[
\det (P_N(m))=\prod_{k \neq m}x_{j(k)}^{-2}\prod_{1\leq r <s \leq N}^{m'}\left[ \dfrac{1}{x_{j(s)}}-\dfrac{1}{x_{j(r)}} \right],
\]
where the $m'$ means we have deleted all of the terms with $x_{j(m)}$.  This leaves us with
\begin{align*}
{\bf b_N}(m)=&x_{j(m)}^{2}\prod_{k>m}\left[\dfrac{1}{x_{j(k)}}-\dfrac{1}{x_{j(m)}}   \right]^{-1}\prod_{l<m}\left[\dfrac{1}{x_{j(m)}}-\dfrac{1}{x_{j(k)}}\right]^{-1}\\
=&(-1)^{N+m}x_{j(m)}^{N+1}\prod_{k \neq m} \left[1-\dfrac{x_{j(m)}}{x_{j(k)}}  \right]^{-1} 
\end{align*}
\end{proof}
\noindent These coefficients have the property that
\[
{\bf b_N}(m)x_{j(m}^{-(N+2)}=O(\dfrac{1}{x_{j(1)}}).
\]
This allows us to get close to $A_m(x)$, since
\begin{align*}
&\sum_{m=1}^{N}\dfrac{{\bf b_N}(m)}{\alpha^2+(x-x_{j(m)})^2}\\
=&\sum_{m=1}^{N}{\bf b_N}(m)x_{j(m)}^{-2}\left[A_0(x)+\dfrac{A_1(x)}{x_{j(m)}}+\cdots+\dfrac{A_{N-1}(x)}{x_{j(m)^{N-1}}}+\cdots  \right]\\
=&A_{N-1}(x) + O(\dfrac{1}{x_{j(1)}})
\end{align*}
\begin{prop}
If $p(x)\in\Pi[x]$ and $\epsilon>0$, then there exists an $N\geq 1$ and a sequence $\{b_m\}_{m=1}^{N}$ such that 
\[
\sup_{x\in [a,b]}\left|p(x)-\sum_{m=1}^{N}\dfrac{b_m}{\alpha^2+(x-x_{j(m)})^2}\right|<\epsilon.
\]
\end{prop}
\begin{proof}
Let $N=\deg(p)$.  Then we may expand $p(x)$ in terms of $\{A_k(x)\}$, that is,
\[
p(x)=\sum_{k=0}^{N}c_{k}A_{k}(x).
\]
Then the coefficients that we need are a linear combination of the ones we found above.
\[
b_m=\sum_{k=m}^{N+1}c_{k-1}{\bf b_{k}}(m)
\]
From this we see that
\[
\sum_{m=1}^{N+1}\dfrac{b_m}{\alpha^2+(x-x_{j(m)})^2} = p(x) + O(\dfrac{1}{x_{j(1)}})
\]
We need only take $x_{j(1)}$ so large that the error term falls below $\epsilon$.
\end{proof}
Finally, we are in position to prove our main result.
\begin{proof}[Proof of Theorem]  
Let $\epsilon>0$, and $f(x)$ be given, then by the Stone-Weierstrass theorem, we may find a polynomial $p(x)$ such that
\[
\sup_{x\in[a,b]}\left|f(x)-p(x)  \right|<\dfrac{\epsilon}{2}.
\]
The above proposition allows us to find $\{b_m\}$ such that
\[
\sup_{x\in[a,b]}\left|p(x)- \sum_{m=1}^{N+1}\dfrac{b_m}{\alpha^2+(x-x_{j(m)})^2} \right|<\dfrac{\epsilon}{2}
\]
The triangle inequality finishes the proof, since
\begin{align*}
&\sup_{x\in[a,b]}\left|f(x)- \sum_{m=1}^{N+1}\dfrac{b_m}{\alpha^2+(x-x_{j(m)})^2} \right| \\
\leq &\sup_{x\in[a,b]}\left| f(x)-p(x) \right| +\sup_{x\in[a,b]}\left|p(x)- \sum_{m=1}^{N+1}\dfrac{b_m}{\alpha^2+(x-x_{j(m)})^2} \right|\\
\leq& \dfrac{\epsilon}{2} +\dfrac{\epsilon}{2} =\epsilon
\end{align*}
\end{proof}

\end{document}